\documentclass[12pt]{article}
\usepackage{amsmath, amsfonts, amsthm, amssymb, mathrsfs}

\newcommand{\h}{\hat{h}}
\newcommand{\hlambda}{\hat{\lambda}}
\newcommand{\QQ}{\mathbb{Q}}
\newcommand{\CC}{\mathbb{C}}
\newcommand{\PP}{\mathbb{P}}
\newcommand{\RR}{\mathbb{R}}
\newcommand{\ZZ}{\mathbb{Z}}
\newcommand{\Bcal}{\mathcal{B}}
\newcommand{\Ocal}{\mathcal{O}}
\newcommand{\Rcal}{\mathcal{R}}
\newcommand{\scD}{\mathscr{D}}
\renewcommand{\epsilon}{\varepsilon}

\newtheorem{proposition}{Proposition}
\newtheorem{theorem}[proposition]{Theorem}
\newtheorem{lemma}[proposition]{Lemma}
\newtheorem*{definition}{Definition}

\theoremstyle{remark}
\newtheorem*{remark}{Remark}
\newtheorem*{ack}{Acknowledgements}

\title{Lower bounds on the canonical height associated to the morphism $\phi(z)= z^d+c$\\ {\small (Short title: Lower bounds on a certain canonical height)}}
\author{Patrick Ingram\footnote{The author was supported by a PDF grant from NSERC.\newline { Mathematics Subject Classification: 11G50, 37F10, 11G99}\newline Keywords: canonical height, arithmetic dynamics, iterates of polynomials}}

\begin{document}

\maketitle

\begin{abstract}
A lower bound is obtained of the canonical height associated to the morphism $\phi(z)=z^d+c$ evaluated at wandering points $\alpha$.  The lower bound is of the form $Ch(c)$, for some constant $C$ depending on the the number of primes of bad reduction for $\phi$, and the degree of the number field $\QQ(c, \alpha)$.
\end{abstract}

Suppose that $\phi:\PP^1\rightarrow\PP^1$ is a rational morphism of degree at least 2, defined over a number field $K$.  One may associate to $\phi$, by a construction due essentially to Tate (see  \cite{cs} for a general construction), a \emph{canonical height} function $\h_\phi:\PP^1(\overline{K})\rightarrow\RR$.  This height function satisfies the two properties
$$\h_\phi(\phi(\alpha))=\deg(\phi)\h_\phi(\alpha)\qquad\text{and}\qquad \h_\phi(\alpha)=h(\alpha)+O(1),$$
where $h$ is the usual absolute logarithmic height.  It is reasonably easy to show, from the two properties above, that $\h_\phi(\alpha)=0$ just in case $\phi^j(\alpha)=\phi^i(\alpha)$ for some $j\neq i$, that is, just in case $\alpha$ is a \emph{pre-periodic} point for $\phi$.  It is natural to ask how small the value of $\h_\phi$ can be at $K$-rational points which are not pre-periodic, i.e., at \emph{wandering points} for $\phi$.  We will examine this question for morphisms of the form $\phi(z)=z^d+c$, with $d\geq 2$, and are interested primarily in the dependence on $c$ (rather than the dependence on $[K:\QQ]$; see \cite{baker}).

The canonical heights mentioned above are analogous to the canonical heights on elliptic curves (and more general abelian varieties) studied by N\'eron and Tate.  The analogous question in this context, namely how small the canonical height of a  non-torsion rational point on an elliptic curve may be, is the subject of a conjecture of Lang.  Specifically, Lang conjectured that the height of a point of infinite order in $E(K)$ is bounded below by a constant multiple of $\max\{h(j_E), \log |\operatorname{Norm}_{K/\QQ}\scD_{E/K}|, 1\}$, where $j_E$ and $\scD_{E/K}$ are the $j$-invariant and minimal discriminant of $E/K$ respectively (see \cite{aec} for definitions of these terms).  Silverman~\cite{js:lang} has given a partial solution to this conjecture, proving that (for a non-torsion point $P$ on an elliptic curve $E$)
$$\h(P)\geq C\max\{h(j_E), \log |\operatorname{Norm}_{K/\QQ}\scD_{E/K}|, 1\},$$
where $C$ depends on $[K:\QQ]$, as well as the number of primes at which $E$ has split multiplicative reduction.

In the case of a morphism $\phi:\PP^1\rightarrow\PP^1$ defined over a number field $K$, Silverman \cite[p.~221]{js:ads} has made a conjecture analogous to Lang's, which suggests a lower bound on the canonical height of $K$-rational wandering points depending on various data related to $\phi$.  For our morphisms $\phi(z)=z^d+c$, a reasonable version of this conjecture would be that for any wandering point $\alpha\in\PP^1(K)$,
$$\h_\phi(\alpha)\geq C\max\{h(c), 1\}$$
for some constant $C$ depending only on $d$ and the field $K$.
  We will prove a weaker version of this claim, similar in flavour to the result of Silverman mentioned above.  
  In subsequent work, Hindry and Silverman \cite{hs} showed that Lang's Conjecture for elliptic curves, in full uniformity, follows from the $abc$ Conjecture of Masser and Oesterl\'{e}.  It would be interesting to see if such an assumption would give a uniform result in this case as well, although there are several features of elliptic curves (in particular, the Kodaira-N\'{e}ron classification of types of bad reduction) exploited by Hindry and Silverman which have no obvious analogue in this setting.
  
    Throughout, we will consider $\phi$ simply as an affine map, as the point at infinity is a totally ramified fixed point of $\phi$.  We suppose that every non-archimedean valuation $v\in M_K$ is normalized, so that $v(\pi)=1$ for $\pi$ a uniformizer of the prime corresponding to $v$.  The places of bad reduction for $\phi$, in the sense of \cite[Chapter~2]{js:ads}, are precisely those for which $v(c)<0$.  We will distinguish further between bad primes of \emph{Type II}, where $v(c)$ is divisible by $d$, and bad primes of \emph{Type I}, where this is not the case.  Our main result is the following:

\begin{theorem}
Let $K$ be a number field, and $c\in K$.  Then there is a constant $C>0$ depending only on $d$, the degree $[K:\QQ]$, and the number of Type II places for $\phi(z)=z^d+c$, such that for all wandering points $\alpha\in K$,
$$\h_\phi(\alpha)\geq C\max\{h(c), 1\}.$$
\end{theorem}

In particular, if $c$ is an algebraic integer, then $C$ depends only on $d$ and $[K:\QQ]$.
This is actually a slightly weaker form of the more technical Theorem~\ref{mainresult}.  Note that for any $M$, there are only finitely many $c\in\overline{\QQ}$ such that $h(c)<M$ and $[\QQ(c):\QQ]<M$.  In particular, there are only finitely many values of $c$ for which the estimate below is trivial, and the constants may be adjusted for these values to produce Theorem~1.

\begin{theorem}\label{mainresult}
Let  $c\in K$, let $s$ be the number of Type II places for $\phi(z)=z^d+c$,  let $r\leq [K:\QQ]$ be the number of distinct archimedean valuations on $K$ (i.e., the number of real embeddings plus half the number of complex embeddings), let $$m=\begin{cases} 2 &\text{if }d=2\\ 1&\text{otherwise,}\end{cases}$$ and let 
$$N=\frac{2}{d^m}\left((d^m+1)^{r+s+1}-1\right).$$
Then for any $\alpha\in K$, either
$\phi^i(\alpha)=\phi^j(\alpha)$
for some $i\neq j< N$, or else
$$\h_\phi(\alpha)\geq \frac{1}{d^{N+2}} \left(h(c)-d(d+2m)\log 2\right).$$
\end{theorem}

Theorem~\ref{mainresult} also yields a bound on the period of pre-periodic points of $\phi$ which depends only on $d$, $[K:\QQ]$, and the number places at which $\phi$ has Type II reduction.  Benedetto~\cite{benedetto} has proven a general result for pre-periodic points of polynomials which provides a bound on the number of pre-periodic points for $\phi$ which is much stronger, both in terms of the dependence on $[K:\QQ]$, and the dependence on the number of bad primes.  Benedetto's bound depends on the number of primes of bad reduction, and so one might suspect that the above bound (depending only the number of primes of Type II reduction) could be stronger in some cases. But we will see below that $\phi$ can have no pre-periodic points whatsoever if there is a prime of Type I reduction, and so Benedetto's bound is stronger in every non-trivial case.

The proof of Theorem~\ref{mainresult} is largely motivated by Silverman's result for elliptic curves \cite{js:lang}.  That result is proven by expressing the canonical height as a sum of local heights, and bounding each from below.  It turns out that the heights corresponding to primes of good, additive, or non-split multiplicative reduction are relatively easy to estimate if $P\in 12E(K)$, and so starting with that condition one need only worry about archimedean valuations, and those corresponding to primes of split multiplicative reduction for $E$.   While the heights at these valuations are more complicated, Silverman employs a pigeon-hole argument to show that any point $P\in E(K)$ must have some multiple $nP$, with $n$ not too large, whose height at each of these places may be estimated in the fashion required.  Thus a lower bound is obtained on the height of some multiple of $P$, and the bound on the height of $P$ is recovered by the transformation law of the canonical height: $\h(nP)=n^2\h(P)$.

The proof of Theorem~\ref{mainresult} proceeds along similar lines, but there is an added difficulty: without the underlying group structure, it is hard to produce any lower bound on the local heights at the badly behaved primes (in this context, archimedean primes and those at which $\phi$ has reduction of Type II).  The solution to this added complication is motivated by a simple `gap principle':  that two rational numbers of small height cannot be too close together in the usual metric (unless they are equal).  Indeed, if $a_1b_2\neq b_1a_2$, then the product rule for valuations (or a more elementary argument) tells us that
$$-\log\left|\frac{a_1}{b_1}-\frac{a_2}{b_2}\right|=\sum_{v<\infty}\log\left|\frac{a_1}{b_1}-\frac{a_2}{b_2}\right|_v\leq\log|b_1|+\log|b_2|,$$
and so an upper bound on the archimedean distance between the two points gives a lower bound on the height of one.  Prompted by this observation, we shall show that if none of the iterates $\phi^j(z)$, with $j$ reasonably small, have large height at a given valuation, then several must be extremely close together in the relevant metric.  Using a pigeonhole argument and the product rule, we obtain the bound on the global canonical height.

In the final section, we will turn to more computational matters.  Theorem~\ref{mainresult} indicates that there is an absolute constant $A>0$ such that if $\phi(z)=z^2+c$, with $c\in\ZZ$, then 
$$\inf_{\alpha\in\QQ}\h_\phi(\alpha)\geq A\max\{\log|c|, 1\},$$
where the infimum is taken over wandering points.  The value in the theorem, however, is far from optimal, and is in fact trivial for small values of $|c|$.  We remedy this by computing the minimum value of $\h_\phi(\alpha)$ in the remaining cases, and provide a small amount of computational evidence towards the `true' value of the constant $A$.  In particular, we show that
$$\h_\phi(\alpha)\geq \frac{1}{32}\max\{\log|c|, 1\},$$
if $\alpha$ is not pre-periodic for $\phi$, and explicitly describe the possible pre-periodic points.
Further, we construct a family of examples to demonstrate that
$$\frac{1}{32}\leq \liminf_{c\rightarrow-\infty}\left(\inf_{\alpha\in \QQ}\frac{\h_\phi(\alpha)}{\log|c|}\right)\leq \frac{1}{8}$$
(where again we omit pre-periodic points).
Computation indicates that the upper bound better reflects the truth than does the lower bound.  Note that the problem is essentially trivial for $c>0$, and it turns out that
$$\lim_{c\rightarrow\infty}\left(\inf_{\alpha\in\QQ}\frac{\h_\phi(\alpha)}{\log|c|}\right)=\frac{1}{2}$$


\begin{ack}
The author would like to thank Joseph H.~Silverman, Matthew Baker, Michael Zieve, and an anonymous referee for several useful comments on, and corrections to, earlier drafts of this paper.
\end{ack}

\section{Valuations and heights}

Let $K$ be a number field, and let $M_K$ be the usual set of valuations on $K$, with $M_K^0$ denoting the set of nonarchimedean valuations, and $M_K^\infty$ the set of archimedean ones.  To each valuation $v\in M_K$ we attach a naive local height
$$\lambda_v(\alpha)=\max\{\log|\alpha|_v, 0\},$$
and a canonical height for $\phi(z)=z^d+c$ (for $d\geq 2$),
$$\hlambda_{\phi, v}(\alpha)=\lim_{k\rightarrow\infty}\frac{\lambda_v(\phi^k(\alpha))}{d^k}.$$
As noted in \cite{js:ads}, this limit is not guaranteed to exist for an arbitrary morphism $\phi:\PP^1\rightarrow\PP^1$, but is if $\phi$ is a polynomial.
If $n_v=[K_v:\QQ_v]$ is the local degree of $K$, then the usual absolute height of $\alpha\in K$ is given by
\begin{alignat*}{1}
h(\alpha)&=\frac{1}{[K:\QQ]}\sum_{v\in M_K}n_v\lambda_v(\alpha)\\
\intertext{and the canonical height relative to $\phi$ is}
\h_\phi(\alpha)&=\frac{1}{[K:\QQ]}\sum_{v\in M_{K}}n_v\hlambda_{\phi, v}(\alpha).
\end{alignat*} 
 We will also use the notation $\Ocal_\phi^+(\alpha)$ for the forward orbit of $\alpha$ under $\phi$, i.e., the set $\{\alpha, \phi(\alpha), \phi^2(\alpha), \ldots\}$.
 
As mentioned in the introduction, we will distinguish between two types of primes of bad reduction.
\begin{definition}
Suppose that $v:K^*\twoheadrightarrow\ZZ$ is a normalized (as in the introduction) nonarchimedean valuation.  Then we will say that $\phi$ has \emph{Type I} reduction at $v$ if $v(c)<0$ and $v(c)\not\in d\ZZ$.  We will say that $\phi$ has \emph{Type II} reduction at $v$ if $v(c)<0$ and $v(c)\in d\ZZ$.  As usual, $\phi$ has good reduction at $v$ if $v(c)\geq 0$.
\end{definition}

It should be noted that, while we are using the definition of ``good reduction'' used by Silverman \cite[p.~58]{js:ads}, there is an alternate definition suggested by Szpiro and Tucker \cite[Definition~1.1]{st2}. 
Let $U$ be the support of the ramification divisor of a rational function $\phi:\PP^1\rightarrow\PP^1$.  Then $\phi$ is said to have \emph{critically good reduction} at $v$ if  for any distinct points $P$ and $Q\in U$, or any distinct points $P$ and $Q\in \phi(U)$, the images of $P$ and $Q$  modulo the maximal ideal in the ring of $v$-integers of $K$ are distinct.  As the ramification divisor of $\phi(z)=z^d+c$ is supported on $U=\{0, \infty\}$, and as $\phi(U)=\{c, \infty\}$, we can see that the primes at which this particular $\phi$ has critically bad reduction are precisely the primes modulo which $c$ and $\infty$ coincide.  Thus, for our functions, the definitions are equivalent.

  In  the argument below, primes of Type I reduction (and those of good reduction) are those for which local heights may be easily estimated, in the same sense that local heights on elliptic curves at primes of additive or non-split multiplicative reduction (and  at those of good reduction) may be easily estimated.  The primes of Type II will be the more problematic primes, playing a role similar to that played by primes of split multiplicative reduction in \cite{js:lang}.
The analogy extends slightly: the property of being a Type II prime (like that of being a prime of split multiplicative reduction) is stable under field extension, while the property of being a Type I prime is not (for example, $3$ is a Type I prime for $\phi(z)=z^2+\frac{1}{3}$ over $\QQ$, but a Type II prime for the same morphism over $\QQ(\sqrt{3})$).  Of course, the distinction between the different types of bad reduction for elliptic curves is one of deep arithmetic significance; it is unclear whether the distinction introduced here is a special case of an arithmetically natural classification types of bad reduction for morphisms, or merely an \emph{ad hoc} division which is convenient in this particular case.



\section{The archimedean places}\label{arch}

For the present section we fix a valuation $v\in M_K^\infty$.  We will also fix an embedding $\overline{K}\rightarrow \CC$ corresponding to this absolute value, allowing us to speak unambiguously about $|\cdot|_v$-values of elements of $\overline{K}$.

It is convenient to think of iterates escaping a set of small local height, and so we will define a set of points with reasonably large local height.
For each archimedean place $v\in M_K^\infty$ we will set
$$\Bcal_v=\left\{\alpha\in K:\lambda_v(\alpha)> \frac{1}{d}\lambda_v(c)+\log 2\right\}.$$
Note that $\Bcal_v$ is, in the notation of \cite[Chapters 2 and 3]{js:ads} a subset of $\Bcal_\phi(\infty)$, the $v$-adic attracting basin of the fixed point $\infty$, and so its complement contains the $v$-adic filled Julia set of $\phi$.

\begin{lemma}\label{lem:arch:bcalv}
Let $\alpha\in\Bcal_v$.  Then $\Ocal_\phi^+(\alpha)\subseteq\Bcal_v$, and
\begin{equation}\label{arch:heightdiff}\frac{\log(1-2^{-d})}{d-1} \leq \hlambda_{\phi, v}(\alpha)-\lambda_v(\alpha)\leq \frac{\log(1+2^{-d})}{d-1}.\end{equation}
For all $j>i\geq 0$ with $\phi^j(\alpha)\neq\phi^i(\alpha)$,
\begin{equation}\label{arch:phidiff}\frac{1}{d}\lambda_v(c)+\log|\phi^j(\alpha)-\phi^i(\alpha)|_v\leq d^{j+1}\hlambda_{\phi, v}(\alpha)+2\log 2.\end{equation}
\end{lemma}

\begin{proof}
Under the assumption that $\alpha\in\Bcal_v$, we have
$$|\phi(\alpha)|_v\geq |\alpha|_v^d-|c|_v> \left(1-\frac{1}{2^d}\right)|\alpha|_v^d>\left(2^d-1\right)\exp(\lambda_v(c)).$$
It follows that
$$\lambda_v(\phi(\alpha))> \lambda_v(c)+\log 3>\frac{1}{d}\lambda_v(c)+\log 2,$$
and so $\phi(\alpha)\in\Bcal_v$.
By induction, $\Ocal_\phi^+(\alpha)\subseteq\Bcal_v$.

We have, for $\alpha\in\Bcal_v$,
$$\left(1-2^{-d}\right)|\alpha|^d_v\leq |\phi(\alpha)|_v\leq |\alpha|_v^d+|c|_v\leq \left(1+2^{-d}\right)|\alpha|^d_v.$$
Taking logarithms and evaluating a telescoping sum, we have
$$\frac{\log\left(1-2^{-d}\right)}{d-1}\leq \hlambda_{\phi, v}(\alpha)-\lambda_v(\alpha)\leq \frac{\log\left(1+2^{-d}\right)}{d-1}.$$

For the final inequality, we have, if $j>i\geq0$ and $\phi^j(\alpha)\neq\phi^i(\alpha)$,
\begin{eqnarray*}
\log|\phi^j(\alpha)-\phi^i(\alpha)|_v&\leq&\log\max\{|\phi^j(\alpha)|_v, |\phi^i(\alpha)|_v\}+\log 2\\
&\leq &\max\{\hlambda_{\phi, v}(\phi^j(\alpha)), \hlambda_{\phi, v}(\phi^i(\alpha))\}+\log 2-\frac{\log\left(1-2^{-d}\right)}{d-1}\\
&\leq & d^j\hlambda_{\phi, v}(\alpha)+2\log 2.
\end{eqnarray*}
By \eqref{arch:heightdiff}, we have
$$\hlambda_{\phi, v}(\alpha)\geq\lambda_v(\alpha)+\frac{\log\left(1-2^{-d}\right)}{d-1} \geq \lambda_v(\alpha)-\log 2>\frac{1}{d}\lambda_v(c),$$
and so we have
$$\frac{1}{d}\lambda_v(c)+\log|\phi^j(\alpha)-\phi^i(\alpha)|_v\leq d^{j+1}\hlambda_{\phi, v}(\alpha)+2\log 2.$$
\end{proof}

\begin{lemma}\label{lem:arch:closetobeta} Let $m=2$ if $d=2$ and $m=1$ otherwise.
If $\phi^m(\alpha)\not\in \Bcal_v$, then for some root $\beta$ of $\phi^m$,
$$\log|\alpha-\beta|_v\leq -\frac{1}{d}\lambda_v(c)+(d+2m-2)\log 2-\log d.$$
\end{lemma}

\begin{proof}
Let $\beta$ be the root of $\phi(z)$ nearest $\alpha$, so that for all roots $\beta'\neq\beta$,
$$|\beta-\beta'|_v\leq |\alpha-\beta|_v+|\alpha-\beta'|_v\leq 2|\alpha-\beta'|_v.$$
Then we have
$$|\phi(\alpha)|_v\geq |\alpha-\beta|_v\prod_{\beta'\neq\beta}\frac{1}{2}|\beta-\beta'|_v=|\alpha-\beta|_v|2^{-d+1}d\beta^{d-1}|_v,$$
where the product is over roots of $\phi(z)$.
As $\phi(\alpha)\not\in\Bcal_v$, and as $|\beta|_v=|c|_v^{1/d}$, it follows that
$$|\alpha-\beta|_v\leq |2^{d-1}d^{-1}|_v|\beta|_v^{1-d}|\phi(\alpha)|_v\leq|2^{d}d^{-1}|_v|c|_v^{-(d-2)/d}.$$
The result follows if $d\geq 3$ and $|c|_v\geq 1$.  If $d\geq 3$ and $|c|_v<1$, then $\phi(\alpha)\not\in\Bcal_v$ implies
$$|\alpha|^d_v\leq |\alpha^d+c|_v+|c|_v\leq |\phi(\alpha)|_v+1\leq 3,$$
and so
$$|\alpha-\beta|_v\leq |\alpha|_v+|\beta|_v\leq 3^{1/d}+|c|_v^{1/d}<\frac{8}{3}\leq 2^{d}d^{-1}\max\{|c|_v , 1\}.$$

If $d=2$ then, replacing $\alpha$ by $\phi(\alpha)$ and $\beta$ by $\gamma$, we have shown that
$$|\phi(\alpha)-\gamma|_v\leq |2|_v$$
for some root $\gamma$ of $\phi(z)$.
Proceeding in the same fashion as before, let $\phi(\alpha)-\gamma=(\alpha-\beta)(\alpha+\beta)$, and assume without loss of generality that $|\alpha-\beta|_v\leq|\alpha+\beta|_v$.  Then
$$|2\beta|_v\leq |\alpha-\beta|_v+|\alpha+\beta|_v\leq 2|\alpha+\beta|_v,$$
and so
$$|2|_v\geq |\phi(\alpha)-\gamma|_v=|\alpha-\beta|_v|\alpha+\beta|_v\geq |\alpha-\beta|_v|\beta|_v.$$
Note that, up to the choice of branch for the square root function, we have
$$\beta=\pm\sqrt{-c+\sqrt{-c}}.$$
If $|c|_v\leq 2$, then $|\beta|_v< 2$, and so (recalling that $\alpha\not\in\Bcal_v$)
$$|\alpha-\beta|_v< 2\sqrt{2}+2< 4\sqrt{2}\leq 8|c|_v^{-1/2}.$$

If, on the other hand, $|c|_v>2$, we have
$$\beta=\pm\sqrt{-c}\left(1+\frac{w}{2}-\frac{w^2}{8}+\cdots\right)$$
where $w^2=-1/c$.  In particular $|w|_v\leq 2^{-1/2}$, and so
$$|\beta|_v\geq\frac{1}{2} |c|_v^{1/2}$$
(obtained by finding the minimum value of the holomorphic function defined by the power series above).
Thus
$$|\alpha-\beta|_v\leq 4|c|_v^{-1/2}.$$

\end{proof}

\begin{lemma}\label{lem:arch:pigeon}
Let $\alpha\in K$, set $m=2$ if $d=2$ and $m=1$ otherwise,  and
suppose that $X$ is a finite set of non-negative integers.  Then there is a subset $Y\subseteq X$ containing at least $\frac{1}{d^m+1}(\#X-m)$ values such that for all $j,i\in Y$ with $j>i$ and $\phi^j(\alpha)\neq\phi^i(\alpha)$, we have
$$\frac{1}{d}\lambda_v(c)+\log|\phi^j(\alpha)-\phi^i(\alpha)|_v\leq d^{j+1}\hlambda_{\phi, v}(\alpha)+(d+2m-1)\log 2-\log d.$$
\end{lemma}

\begin{proof}
Suppose that at least $\frac{1}{d^m+1}(\#X-m)+m$ values $k\in X$ have $\phi^k(\alpha)\in \Bcal_v$, and let $Y$ be the set of such values.  If $j,i\in Y$ with $j>i$ and $\phi^j(\alpha)\neq\phi^i(\alpha)$, then \eqref{arch:phidiff} of Lemma~\ref{lem:arch:bcalv} applied to $\phi^i(\alpha)\in\Bcal_v$ implies
\begin{eqnarray*}
\frac{1}{d}\lambda_{v}(c)+\log|\phi^j(\alpha)-\phi^i(\alpha)|_v&\leq& d^{j-i+1}\hlambda_{\phi, v}(\phi^i(\alpha))+2\log 2\\
&<& d^{j+1}\hlambda_{\phi, v}(\alpha)+(d+2m-1)\log 2-\log d.
\end{eqnarray*}

So suppose that fewer than $\frac{1}{d^m+1}(\#X-m)+m$ values $k\in X$ witness $\phi^k(\alpha)\in\Bcal_v$.  Thus there are more than $\frac{d^m}{d^m+1}(\#X-m)+m$ values $k\in X$ such that $\phi^k(\alpha)\not\in\Bcal_v$, and so more than $\frac{d^m}{d^m+1}(\#X-m)$ such that $\phi^{k+m}(\alpha)\not\in\Bcal_v$.

By Lemma~\ref{lem:arch:closetobeta}, and the pigeon-hole principle, there is $\beta\in\overline{K}$ with $\phi^m(\beta)=0$ and
$$\log|\phi^k(\alpha)-\beta|_v<-\frac{1}{d}\lambda_v(c)+(d+2m-2)\log 2-\log d$$
for at least $\# (X-m)/(d^m+1)$ values $k\in X$.  If $j$ and $i$ are two of these values, then by the triangle inequality
\begin{eqnarray*}
\log|\phi^j(\alpha)-\phi^i(\alpha)|_v&\leq& \log \left(2\max\{|\phi^j(\alpha)-\beta|_v, |\phi^i(\alpha)-\beta|_v\}\right)\\
&<& -\frac{1}{d}\lambda_v(c)+(d+2m-1)\log 2-\log d.
\end{eqnarray*}
As $\hlambda_{\phi, v}(\alpha)\geq 0$, we have
$$\frac{1}{d}\lambda_v(c)+\log|\phi^j(\alpha)-\phi^i(\alpha)|_v\leq d^{j+1}\hlambda_{\phi, v}(\alpha)+(d+2m-1)\log 2-\log d.$$
\end{proof}


\section{Non-archimedean places}

For this section we fix a valuation $v\in M_K^0$, and  an extension of $|\cdot|_v$ to the algebraic closure of $K$.
As in the previous section, we will define a set of points with reasonably large local height
$$\Bcal_v=\left\{\alpha\in K: \lambda_v(\alpha)>\frac{1}{d}\lambda_v(c)\right\}.$$
 We will also define a `boundary' to the above set,
 $$\Rcal_v=\left\{\alpha\in K:\lambda_v(\alpha)=\frac{1}{d}\lambda_v(c)\right\}.$$
Note that $\Rcal_v$ is empty if $\phi$ has Type I reduction at $v$, a point which greatly simplifies this case.   It should also be noted that the $v$-adic filled Julia set of $\phi$ is entirely contained in $\Rcal_v$.

\begin{lemma}\label{nonarchbv}
If $\alpha\in\Bcal_v$, then $\Ocal_\phi^+(\alpha)\subseteq\Bcal_v$.  Furthermore, $\alpha\in \Bcal_v$ implies
\begin{equation}\label{nonarch:heightdiff}\hlambda_{\phi, v}(\alpha)=\lambda_v(\alpha)\end{equation}
and for all $j>i\geq 1$,
$$\frac{1}{d}\lambda_v(c)+\log|\phi^j(\alpha)-\phi^i(\alpha)|_v\leq d^{j+1}\hlambda_{\phi, v}(\alpha).$$
\end{lemma}

Note that Lemma~\ref{lem:arch:bcalv} and Lemma~\ref{nonarchbv} immediately imply  that
$$\alpha\in\bigcap_{v\in M_K}\Bcal_v\Longrightarrow \h_\phi(\alpha)\geq \frac{1}{d}h(c).$$

\begin{proof}[Proof of Lemma~\ref{nonarchbv}]
The the condition $\alpha\in \Bcal_v$ tells us that
$|\alpha|_v^d>\max\{|c|_v, 1\}$,
and it follows, by the ultra-metric inequality, that
$$|\phi(\alpha)|_v=\max\{|\alpha^d|_v, |c|_v\}=|\alpha|_v^d>|\alpha|_v.$$
From this it follows that $\phi(\alpha)\in\Bcal_v$, and by induction $\Ocal_\phi^+(\alpha)\subseteq\Bcal_v$.  Induction also shows that
$$\log|\phi^k(\alpha)|_v=d^k\log|\alpha|_v>d^{k-1}\lambda_v(c),$$
from which we have immediately $\hlambda_{\phi, v}(\alpha)=\lambda_v(\alpha)$.
Finally, as $|\alpha|_v>1$ for $\alpha\in\Bcal_v$, we have $|\phi^j(\alpha)|_v>|\phi^i(\alpha)|_v>1$ for all $j>i$, and so
\begin{eqnarray*}
\frac{1}{d}\lambda_v(c)+\log|\phi^j(\alpha)-\phi^i(\alpha)|_v&=&\frac{1}{d}\lambda_v(c)+\log|\phi^j(\alpha)|_v\\
&=&\frac{1}{d}\lambda_v(c)+\hlambda_{\phi, v}(\phi^j(\alpha))\qquad\text{by \eqref{nonarch:heightdiff}}\\
&<& (d^j+1)\hlambda_{\phi, v}(\alpha)\quad\qquad\text{as }\alpha\in\Bcal_v\\
&\leq&d^{j+1}\hlambda_{\phi, v}(\alpha).
\end{eqnarray*}
\end{proof}

\begin{lemma}\label{lemma:rv}
If $\alpha\not\in \Rcal_v$, then $\phi(\alpha)\in \Bcal_v$.
\end{lemma}

\begin{proof}
If $\alpha\in\Bcal_v$ then we're done by the previous lemma, so suppose that $\alpha$ is in neither $\Bcal_v$ nor $\Rcal_v$, implying $\lambda_v(\alpha)<\frac{1}{d}\lambda_v(c)$. Then $|\alpha|_v^d<|c|_v$, and so
$$|\phi(\alpha)|_v=\max\{|\alpha|_v^d, |c|_v\}=|c|_v.$$
We have
$$\lambda_v(\phi(\alpha))=\lambda_v(c)>\frac{1}{d}\lambda_v(c),$$
as the condition $\lambda_v(c)>d\lambda_v(\alpha)\geq 0$ ensures that $\lambda_v(c)$ is strictly positive.
\end{proof}

Thus, the problematic points, in the nonarchimedean places, are those in $\Rcal_v$.  If $v(c)\in d\ZZ$, it is quite possible to construct arbitrarily long chains with the property
$$\phi^k(\alpha)\in\Rcal_v$$ for all $k\leq n$,
with $\phi^{n+1}(\alpha)\not\in\Rcal_v$ (simply by choosing $\alpha\in\phi^{-n}(0)$).  In particular, it would appear that no purely local argument can give a useful lower bound on the local height $\hlambda_{\phi, v}(\alpha)$. 

Note that Lemma~\ref{lemma:rv} also puts fairly strong restrictions on the pre-periodic points of $\phi$.  In particular, if $\alpha$ is a pre-periodic point of $\phi$, we must have $dv(\alpha)=v(c)$ for every non-archimedean $v\in M_K$.  This generalizes Corollary~4 of \cite{wr}, which makes this assertion for $d=2$ and $K=\QQ$.  In particular, if $\phi$ has Type I reduction at any primes, then $\phi$ has no pre-periodic points in $K$ at all.

Recall that we have fixed an extension of $|\cdot|_v$ to $\overline{K}$.

\begin{lemma}\label{nonarchbeta}
Let $m=2$ if $d=2$ and $m=1$ otherwise.
Suppose that $v$ is a prime of bad reduction, and that $\phi^m(\alpha)\in\Rcal_v$.  Then there is a root $\beta$ of $\phi^m(\alpha)$ in $\overline{K}$ such that
$$\log|\alpha-\beta|_v\leq -\frac{1}{d}\lambda_v(c)-\log\left|d\right|_v-\log|2|_v$$
\end{lemma}

\begin{proof}
As $v$ is a prime of bad reduction, recall that $\lambda_v(c)=\log|c|_v>0$, and
note that by the reasoning in Lemma~\ref{nonarchbv}, we must have $|\beta|_v=|c|_v^{1/d}$ for all $\beta\in\bigcup_{k\geq 1}\phi^{-k}(0)$.  Otherwise we have $|\phi^j(\beta)|>|c|_v$ for all $j\geq 1$, and so in particular $0=|\phi^k(\beta)|_v>|c|_v>1$.
We proceed much as in the archimedean case.

Let $\beta$ be a root of $\phi$ satisfying
$$|\alpha-\beta|_v\leq |\alpha-\beta'|_v$$
for all roots $\beta'$ of $\phi$.  Then for any $\phi(\beta')=0$ we have
$$|\beta-\beta'|_v\leq\max\{|\alpha-\beta|_v, |\alpha-\beta'|_v\}\leq |\alpha-\beta'|_v,$$
and so
$$|\phi(\alpha)|_v=\prod_{\beta'}|\alpha-\beta'|_v\geq|\alpha-\beta|_v\prod_{\beta'\neq\beta}|\beta-\beta'|_v=|\alpha-\beta|_v|d\beta^{d-1}|_v,$$
where the products are over roots of $\phi$.  As $|\beta|_v=|\phi(\alpha)|_v=|c|_v^{1/d}$, we have, for $d\geq 3$,
$$|\alpha-\beta|_v=|d|_v^{-1}|c|_v^{(2-d)/d}<|d|_v^{-1}|c|_v^{-1/d},$$
which is the bound above.

If $d=2$, then we have shown, replacing $\alpha$ above by $\phi(\alpha)$ and $\beta$ by $\gamma$,  that $\phi^2(\alpha)\in\Rcal_v$ implies 
$$|\phi(\alpha)-\gamma|_v<|2|_v^{-1}$$
for some root $\gamma$ of $\phi(z)$.  Let $\pm\beta$ be the roots of $\phi(z)-\gamma$, and suppose without loss of generality that
$$|\alpha-\beta|_v\leq |\alpha+\beta|_v.$$
Then $$|2|_v^{-1}\geq |\phi(\alpha)-\gamma|_v=|\alpha-\beta|_v|\alpha+\beta|_v\geq |\alpha-\beta|_v|2\beta|_v,$$
and hence
$$|\alpha-\beta|_v\leq |2|_v^{-2}|c|_v^{-1/2}.$$
\end{proof}

\begin{lemma}\label{lem:nonarch:pigeon}
Let $\alpha\in K$, set $m=2$ if $d=2$ and $m=1$ otherwise, and
suppose $v$ is a prime of bad reduction for $\phi$.   If $X$ is a finite set of non-negative integers, then there is a subset $Y\subseteq X$ containing at least $(\#X-m)/(d^m+1)$ values such that for all $j>i$ in $Y$,
$$\frac{1}{d}\lambda_v(c)+\log|\phi^j(\alpha)-\phi^i(\alpha)|_v\leq d^{j+1}\hlambda_{\phi, v}(\alpha)-\log|d|_v-\log|2|_v.$$
\end{lemma}

\begin{proof}
The lemma follows from the above in exactly the same way that Lemma~\ref{lem:arch:pigeon} follows from the other lemmas of Section~\ref{arch}.
\end{proof}

\begin{lemma}\label{goodred}
Suppose that $v$ is a prime of good reduction.  Then for all $j>i$,
$$\log|\phi^j(\alpha)-\phi^i(\alpha)|\leq d^{j+1}\lambda_{\phi, v}(\alpha).$$
\end{lemma}

\begin{proof}
If $\phi^i(\alpha)\in\Bcal_v$, then this follows from Lemma~\ref{nonarchbv}.  If $\phi^i(\alpha)\not\in\Bcal_v$ then $\log|\phi^i(\alpha)|_v\leq 0$.  If $\phi^j(\alpha)\in\Bcal_v$ then
$$\log|\phi^j(\alpha)-\phi^i(\alpha)|_v=\log|\phi^j(\alpha)|_v=\hlambda_{\phi, v}(\phi^j(\alpha))=d^j\hlambda_{\phi, v}(\alpha).$$
If $\phi^j(\alpha)\not\in\Bcal_v$, then
$$\log|\phi^j(\alpha)-\phi^i(\alpha)|_v\leq\log\max\{|\phi^j(\alpha)|_v, |\phi^i(\alpha)|_v\}\leq 0\leq d^j\hlambda_{\phi, v}(\alpha).$$
\end{proof}


\section{Proof of Theorem~\ref{mainresult}}

Let $K$, $\phi$, $N$, etc. be as in the statement of the result.   To begin, we will assume that $\alpha\in\phi(K)$.

Let $X=\{0, 1, \ldots, N-1\}$, and let $v\in M_K$ be either an archimedean valuation or a valuation at which $\phi$ has reduction Type II.  For convenience, write
$$\delta_v=\begin{cases}(d+2m-1) \log2-\log d & \text{if }v\in M_K^{\infty}\\ -\log|d|_v-\log|2|_v& \text{if }v\in M_K^0.\end{cases}$$
  By Lemma~\ref{lem:arch:pigeon} or Lemma~\ref{lem:nonarch:pigeon}, we may choose a subset $X'\subseteq X$ such that
$$\# X' \geq \frac{1}{(d^m+1)}(\# X-m)\geq\frac{2}{d^m}((d^m+1)^{r+s}-1)$$
and such that for all $j>i\in X'$, we have
\begin{equation}\label{main}\frac{1}{d}\lambda_v(c)+\log|\phi^j(\alpha)-\phi^i(\alpha)|_v\leq d^{j+1}\hlambda_{\phi, v}(\alpha)+\delta_v.\end{equation}
Proceeding by induction, we have a set $Y\subseteq X$ with
$$\# Y\geq \frac{2}{d^m}((d^m+1)^1-1)=2$$
such that \eqref{main} holds for every valuation in $M_K^\infty$ or at which $\phi$ has Type II reduction (for all $j>i$ in $Y$).  For each of the remaining valuations $v$ of bad reduction we have supposed that $\alpha\in\phi(K)\subseteq\Bcal_v$, and hence by Lemma~\ref{nonarchbv} have \eqref{main} again.  Finally, for primes $v$ of good reduction, Lemma~\ref{nonarchbv} or Lemma~\ref{goodred} gives us \eqref{main}, depending on whether or not $\alpha\in\Bcal_v$.

Let $j>i$ be two distinct values in $Y$.  Assuming $\phi^j(\alpha)\neq\phi^i(\alpha)$, the product rule gives us (recalling that $n_v=[K_v:\QQ_v]$)
$$\sum_{v\in M_K}n_v\log|\phi^j(\alpha)-\phi^i(\alpha)|_v=0,$$
and so summing \eqref{main} over all valuations  gives us
\begin{eqnarray*}
[K:\QQ]\frac{1}{d}h(c)&=&\sum_{v\in M_K}n_v\left(\frac{1}{d}\lambda_v(c)+\log|\phi^j(\alpha)-\phi^i(\alpha)|_v\right)\\
&\leq &\sum_{v\in M_K}n_v\left(d^{j+1}\hlambda_{\phi, v}(\alpha)+\delta_v\right)\\
&= &[K:\QQ]d^{j+1}\h_\phi(\alpha)+\sum_{v\in M_K}n_v\delta_v\\
&=& [K:\QQ]\left(d^{j+1}\h_\phi(\alpha)+(d+2m)\log 2\right)
\end{eqnarray*}
As $j\leq N-1$, we have
$$\frac{1}{d}h(c)\leq d^{N}\h_\phi(\alpha)+(d+2m)\log 2$$
for points $\alpha\in\phi(K)$.  For points $\alpha\in K$, we may apply this result to $\phi(\alpha)$ to obtain the estimate in Theorem~\ref{mainresult}.

Note that a slightly stronger, and more complicated, statement may be derived by distinguishing completions $K_v$ in which $\phi(z)$ has a root from those in which it has none.


\section{Specific computations for $d=2$ and  $c\in\ZZ$}

Theorem~\ref{mainresult} implies that there is an absolute constant $A>0$  such that if $c\in \ZZ$ and $\alpha\in\QQ$ is not a pre-periodic point for $\phi(z)=z^2+c$, then
\begin{equation}\label{A}\h_\phi(\alpha)\geq A\max\{\log|c|, 1\}.\end{equation}
The theorem, applied directly, allows us to conclude that
$$\h_\phi(\alpha)\geq 2^{-14}(\log|c|-12\log2),$$
but this is far from optimal, and is in fact trivial for $|c|\leq 4096$.  It would be interesting to know how large we may take the constant $A$ in \eqref{A}.

  For $c\geq 1$, one has $\phi(\alpha)>|\alpha|$ for all $\alpha\in\QQ$.  This ensures that $\phi$ can have no pre-periodic points, and makes it essentially trivial to construct a very strong lower bound on $\h_\phi(\alpha)$.  For $c\leq -1$, however, things are slightly more tricky.  Indeed, for negative values of $c$ one may actually encounter pre-periodic points, which have canonical height zero.  For example, we see the pre-periodic structure
$$-\frac{1\pm m}{2}\stackrel{\phi}{\longrightarrow} \frac{1\pm m}{2}\stackrel{\phi}{\longrightarrow} \frac{1\pm m}{2}$$
for $c=(1-m^2)/4$ (with $m\in\ZZ$ odd)
and
$$- \frac{1\pm m}{2}\stackrel{\phi}{\longrightarrow} \frac{1\mp m}{2}\stackrel{\phi}{\longrightarrow} \frac{1\pm m}{2}\stackrel{\phi}{\longrightarrow} \frac{1\mp m}{2}$$
for $c=-(m^2+3)/4$ (with $m\in\ZZ$ odd).

Refining the proof of Theorem~\ref{mainresult} somewhat for this special case, we are able to prove the following:
\begin{proposition}
Let $c\in\ZZ$.  Then for all $\alpha\in\QQ$, if $\alpha$ is a wandering point for $\phi(z)=z^2+c$, then
$$\h_\phi(\alpha)\geq\frac{1}{32}\max\{\log|c|, 1\}.$$
Furthermore, if $\alpha$ is a pre-periodic point for $\phi$, then $\alpha$ and $c$ occur in one of the two families detailed above.
\end{proposition}

Although this bound is certainly an improvement on the blind application of Theorem~\ref{mainresult}, it still seems to be conservative.  In the author's computations, every wandering point $\alpha$ satisfied $\h_\phi(\alpha)\geq\frac{1}{8}\log|c|$.  On the other hand, for any one can construct integers $c<0$ and $\alpha$ with $\h_\phi(\alpha)\leq \frac{1}{8}\log|c|+O(1)$, and one such construction is given below.  It would seem, then, that the `true' value of the constant $A$ in \eqref{A}, if $c$ and $\alpha$ are allowed to be arbitrary, is $1/8$.

\begin{remark}
A similar result may be proven, using almost the exact same argument, for $z^d+c$ with $d\geq 3$ and $c\in\ZZ$.  In particular, if $\alpha\in\QQ$ and $\phi(\alpha)\neq\phi^2(\alpha)$, then
$$\h_\phi(\alpha)\geq\max\left\{\frac{d-2}{d^2}\log|c|+\frac{\log d}{d}-\frac{d+1}{d}\log2, \frac{1}{d^2}\log|c|+\frac{1}{d^2}\log\frac{3}{2}\right\},$$
and so
$$\h_\phi(\alpha)\geq\frac{1}{d^2}\log|c|+O(1)$$
for $|c|$ large enough.  On the other hand, for $|c|\geq 3$, we see that
$$\h_\phi(0)=\frac{1}{d}\h_\phi(c)\leq \frac{1}{d}\log|c|+\frac{1}{d}\log\frac{3}{2}$$
by the estimate on the difference $|\h_\phi(c)-h(c)|$ implied by Lemma~\ref{lem:computations} below.
\end{remark}

\begin{remark}
The classification of pre-periodic behaviour of $\phi(z)=z^2+c$ for $c\in\ZZ$ offered by the proposition is almost certainly not new.  The problem is much harder (indeed, unsolved) for $c\in\QQ$.  There are certainly infinitely many values of $c\in\QQ$ such that $z^2+c$ has a periodic point of period $3$, but an examination of the parametrization of these values \cite[pp.~157--158]{js:ads} confirms that only finitely many may be $S$-integral for any finite set of primes $S$.  Indeed, it is not hard to show (using Theorems~2.21 and 2.28 of \cite{js:ads}) that $\phi(z)=z^2+c$ can have no periodic point of period greater than 4 unless $\mathrm{ord}_2(c)<0$.  There are no values of $c\in\QQ$ such that $\phi$ has a point of period 4 or 5, and it is conjectured that there are no values yielding points of period greater than 3 (see \cite{poonen} and the discussion in \cite[pp.~95--97]{js:ads}).
\end{remark}

The remainder of this section is devoted to the proof of the proposition, as well as to various comments about the computations involved.  We will, throughout, assume that $c$ is an integer.   It is useful to note that in this case $\hlambda_{\phi, v}(\alpha)=\lambda_v(\alpha)$ for all non-archimedean $v$, and so
$$\h_\phi\left(\frac{a}{b}\right)=\hlambda_{\phi, \infty}\left(\frac{a}{b}\right)+\log|b|.$$

First of all, as noted above, the problem is essentially trivial if $c>0$.  Here we have
$|\phi(\alpha)|\geq c$,
for all $\alpha\in\QQ$, and so
$$|\phi^k(\alpha)|\geq c^{2^{k-1}}.$$
Thus
\begin{equation}\label{est1}\h_\phi(\alpha)\geq\hlambda_{\phi, \infty}(\alpha)\geq \frac{1}{2}\log c.\end{equation}
For $c=1$,  we may simply note that
$$|\phi^2(\alpha)|\geq 2,$$
and so $|\phi^k(\alpha)|\geq 2^{2^{k-2}}.$  It follows that
$$\h_\phi(\alpha)\geq \frac{1}{4}\log 2,$$
and so for $c\geq 1$, we have
$$\h_\phi(\alpha)\geq \frac{\log 2}{4}\max\{\log c, 1\}$$
for all $\alpha\in\QQ$.  On the other hand, if $c\geq 5$ then $c\in\Bcal_\infty$, and so Lemma~\ref{lem:arch:bcalv} gives us
$$\h_\phi(c)=\hlambda_{\phi, \infty}(c)\leq \log c+\log \frac{3}{2}.$$
In particular, as $\phi(0)=c$, we have
\begin{equation}\label{cposupper}\h_\phi(0)\leq \frac{1}{2}\log c+\frac{1}{2}\log\frac{3}{2},\end{equation}
showing that one cannot do much better than \eqref{est1} (in particular, the quotient of the upper bound in \eqref{cposupper} and the lower bound in \eqref{est1} is $1+o(1)$ as $c\rightarrow\infty$).

From this point forward, we will restrict attention to negative values $c\in\ZZ$.
 By Lemma~\ref{lem:arch:closetobeta}, if $\phi^2(\alpha)\not\in\Bcal_\infty$, then there is a root $\beta$ of $\phi^2(z)$ such that
 $$\left|\alpha-\beta\right|\leq 8|c|^{-1/2}.$$
 A closer examination of the proof shows that we may take $|\alpha-\beta|<\frac{7}{3}|c|^{-1/2}$ if we stipulate $c\leq -49$.
 As $\phi^2(z)$ has two pairs of roots differing only in sign, the pigeon-hole principle tells us that if $\alpha_0, \alpha_1, \alpha_2\in\QQ$ all satisfy $\phi^2(\alpha_i)\not\in\Bcal_\infty$, then there are values $0\leq i<j\leq 2$ such that
 $$|\alpha_i\pm\alpha_j|\leq \frac{14}{3}|c|^{-1/2}.$$
 In particular, if $\phi^4(\alpha)\not\in\Bcal_\infty$, then there must be two values  $0\leq i<j\leq 2$, such that
 $$|\phi^j(\alpha)\pm\phi^i(\alpha)|\leq \frac{14}{3}|c|^{-1/2}.$$
 Noting that $|\phi_j(\alpha)\pm\phi^i(\alpha)|_v\leq |\phi^j(\alpha)|_v$ for all finite valuations $v$, we have
 either $\phi^j(\alpha)=\pm\phi^i(\alpha)$, or else
 $$\frac{1}{2}\log|c|-\log \frac{14}{3}\leq -|\phi^j(\alpha)-\phi^i(\alpha)|=\sum_{v\neq\infty}|\phi^j(\alpha)-\phi^i(\alpha)|\leq \h_\phi(\phi^j(\alpha)).$$
If the latter holds, then
\begin{equation}\label{est3}\h_\phi(\alpha)\geq \frac{1}{8}\log|c|-\frac{1}{4}\log \frac{14}{3}.\end{equation}
If, on the other hand, $\phi^4(\alpha)\in\Bcal_\infty$, Lemma~\ref{lem:arch:bcalv} provides
$$\hlambda_{\phi, \infty}(\phi^4(\alpha))\geq \frac{1}{2}\max\{\log|c|, 0\},$$
and so
\begin{equation}\label{est2}\h_\phi(\alpha)\geq \frac{1}{32}\log|c|.\end{equation}
For $c\leq -61$, the bound in \eqref{est2} is strictly weaker than that in \eqref{est3}.
Thus, for $c\leq -61$, we have verified that \eqref{est2} holds for $\alpha\in\QQ$ not pre-periodic.

Note that if $\phi^j(\alpha)=\pm\phi^i(\alpha)$, then $\phi^{j+1}(\alpha)=\phi^{i+1}(\alpha)$.  In particular, the above tells us that (still with $c\leq -61$) if $\alpha$ is pre-periodic for $\phi$, then $\alpha$ has period at most 2, although there may be a `tail' before the periodicity.  By solving the equations $\phi(z)=z$ and $\phi^2(z)=z$, we may find all examples of such behaviour, which are just those listed above.

All that remains is to verify the proposition for values $-61\leq c\leq -1$.  As the computation of these canonical heights is not entirely routine, a justification of the accuracy of the computations is in order.  We begin with a lemma, whose proof will be deferred to the end of the section.

\begin{lemma}\label{lem:computations}
Suppose $c\in\ZZ$ is non-zero, and let $\phi(z)=z^d+c$.  Then for all $\alpha\in\QQ$
$$\left|\h_{\phi}(\alpha)-h(\alpha)\right|\leq\frac{\log|2c|}{d-1}.$$
\end{lemma}

Note that by the remarks above, in this case an estimate on the difference between $\h_\phi$ and $h$ is tantamount to an estimate on the difference between $\hlambda_{\phi, \infty}$ and $\lambda_\infty$.

This lemma gives us a reasonably efficient way of computing canonical heights.  If $\epsilon>0$ is any fixed value, we may select $$m=\left\lceil\frac{-\log \epsilon+\log\log|2c|}{\log 2}\right\rceil$$
to ensure that
$$\left|\frac{1}{2^m}\lambda_{\phi, \infty}\left(\phi^m\left(\frac{a}{b}\right)\right)+\log|b|-\h_\phi\left(\frac{a}{b}\right)\right|<\epsilon.$$
In this way we may compute the canonical height of a point to arbitrary accuracy.  Given a single point $\alpha\in\QQ$ of a certain height, the lemma tells us that the only points in $\QQ$ with canonical height strictly less than that of $\alpha$ are those with absolute logarithmic height at most $\h_\phi(\alpha)+\log|2c|$.  Thus if we have a suspected candidate for the point of least canonical height, we may check all points with absolute height less than this bound and thereby find an absolute lower bound on $\h_\phi(\alpha)$. In our computations, some points turn out to have (computed) height less than $\epsilon$, but these turned out in every case to be pre-periodic points in the families above.  Some data from the computation appears in Figure~\ref{figure}, and indicates that the `true' lower bound on $\h_\phi(\alpha)/\log|c|$, for wandering points $\alpha$, may be $1/8$.  Note that the data in Figure~\ref{figure} indicates quite clearly that a better lower bound is available for $\h_\phi$ when the points of period 1 or 2 are $\QQ$-rational.  If the fixed points $\gamma_i$ are in $\QQ$, then they are (by trivial estimates) poorly approximated by $\QQ$-rational wandering points:
$$\left|\frac{a}{b}-\gamma_i\right|\gg \frac{1}{b}\qquad\text{for}\qquad\frac{a}{b}\neq\gamma_i;$$
if $\gamma_i$ are properly quadratic over $\QQ$, this is not the case:
$$\left|\frac{a}{b}-\gamma_i\right|\leq\frac{1}{b^2}\qquad\text{for infinitely many}\qquad\frac{a}{b}\in\QQ.$$
 This is, at least in principle, consistent with the equi-distribution of pre-periodic points of $\phi$ with respect to $\h_\phi$ (see \cite{st}).

\begin{figure}
\begin{center}
\caption{Minimal positive values of $\h_\phi(\alpha)/\log|c|$}\label{figure}
\setlength{\unitlength}{0.240900pt}
\ifx\plotpoint\undefined\newsavebox{\plotpoint}\fi
\sbox{\plotpoint}{\rule[-0.200pt]{0.400pt}{0.400pt}}%
\begin{picture}(1200,600)(0,0)
\sbox{\plotpoint}{\rule[-0.200pt]{0.400pt}{0.400pt}}%
\put(140.0,82.0){\rule[-0.200pt]{4.818pt}{0.400pt}}
\put(120,82){\makebox(0,0)[r]{ 0.1}}
\put(1119.0,82.0){\rule[-0.200pt]{4.818pt}{0.400pt}}
\put(140.0,130.0){\rule[-0.200pt]{4.818pt}{0.400pt}}
\put(120,130){\makebox(0,0)[r]{ 0.2}}
\put(1119.0,130.0){\rule[-0.200pt]{4.818pt}{0.400pt}}
\put(140.0,178.0){\rule[-0.200pt]{4.818pt}{0.400pt}}
\put(120,178){\makebox(0,0)[r]{ 0.3}}
\put(1119.0,178.0){\rule[-0.200pt]{4.818pt}{0.400pt}}
\put(140.0,225.0){\rule[-0.200pt]{4.818pt}{0.400pt}}
\put(120,225){\makebox(0,0)[r]{ 0.4}}
\put(1119.0,225.0){\rule[-0.200pt]{4.818pt}{0.400pt}}
\put(140.0,273.0){\rule[-0.200pt]{4.818pt}{0.400pt}}
\put(120,273){\makebox(0,0)[r]{ 0.5}}
\put(1119.0,273.0){\rule[-0.200pt]{4.818pt}{0.400pt}}
\put(140.0,321.0){\rule[-0.200pt]{4.818pt}{0.400pt}}
\put(120,321){\makebox(0,0)[r]{ 0.6}}
\put(1119.0,321.0){\rule[-0.200pt]{4.818pt}{0.400pt}}
\put(140.0,369.0){\rule[-0.200pt]{4.818pt}{0.400pt}}
\put(120,369){\makebox(0,0)[r]{ 0.7}}
\put(1119.0,369.0){\rule[-0.200pt]{4.818pt}{0.400pt}}
\put(140.0,417.0){\rule[-0.200pt]{4.818pt}{0.400pt}}
\put(120,417){\makebox(0,0)[r]{ 0.8}}
\put(1119.0,417.0){\rule[-0.200pt]{4.818pt}{0.400pt}}
\put(140.0,464.0){\rule[-0.200pt]{4.818pt}{0.400pt}}
\put(120,464){\makebox(0,0)[r]{ 0.9}}
\put(1119.0,464.0){\rule[-0.200pt]{4.818pt}{0.400pt}}
\put(140.0,512.0){\rule[-0.200pt]{4.818pt}{0.400pt}}
\put(120,512){\makebox(0,0)[r]{ 1}}
\put(1119.0,512.0){\rule[-0.200pt]{4.818pt}{0.400pt}}
\put(140.0,560.0){\rule[-0.200pt]{4.818pt}{0.400pt}}
\put(120,560){\makebox(0,0)[r]{ 1.1}}
\put(1119.0,560.0){\rule[-0.200pt]{4.818pt}{0.400pt}}
\put(140.0,82.0){\rule[-0.200pt]{0.400pt}{4.818pt}}
\put(140,41){\makebox(0,0){-70}}
\put(140.0,540.0){\rule[-0.200pt]{0.400pt}{4.818pt}}
\put(283.0,82.0){\rule[-0.200pt]{0.400pt}{4.818pt}}
\put(283,41){\makebox(0,0){-60}}
\put(283.0,540.0){\rule[-0.200pt]{0.400pt}{4.818pt}}
\put(425.0,82.0){\rule[-0.200pt]{0.400pt}{4.818pt}}
\put(425,41){\makebox(0,0){-50}}
\put(425.0,540.0){\rule[-0.200pt]{0.400pt}{4.818pt}}
\put(568.0,82.0){\rule[-0.200pt]{0.400pt}{4.818pt}}
\put(568,41){\makebox(0,0){-40}}
\put(568.0,540.0){\rule[-0.200pt]{0.400pt}{4.818pt}}
\put(711.0,82.0){\rule[-0.200pt]{0.400pt}{4.818pt}}
\put(711,41){\makebox(0,0){-30}}
\put(711.0,540.0){\rule[-0.200pt]{0.400pt}{4.818pt}}
\put(854.0,82.0){\rule[-0.200pt]{0.400pt}{4.818pt}}
\put(854,41){\makebox(0,0){-20}}
\put(854.0,540.0){\rule[-0.200pt]{0.400pt}{4.818pt}}
\put(996.0,82.0){\rule[-0.200pt]{0.400pt}{4.818pt}}
\put(996,41){\makebox(0,0){-10}}
\put(996.0,540.0){\rule[-0.200pt]{0.400pt}{4.818pt}}
\put(1139.0,82.0){\rule[-0.200pt]{0.400pt}{4.818pt}}
\put(1139,41){\makebox(0,0){ 0}}
\put(1139.0,540.0){\rule[-0.200pt]{0.400pt}{4.818pt}}
\put(140.0,82.0){\rule[-0.200pt]{240.659pt}{0.400pt}}
\put(1139.0,82.0){\rule[-0.200pt]{0.400pt}{115.150pt}}
\put(140.0,560.0){\rule[-0.200pt]{240.659pt}{0.400pt}}
\put(140.0,82.0){\rule[-0.200pt]{0.400pt}{115.150pt}}
\put(1110,513){\raisebox{-.8pt}{\makebox(0,0){$\Diamond$}}}
\put(1096,224){\raisebox{-.8pt}{\makebox(0,0){$\Diamond$}}}
\put(1082,141){\raisebox{-.8pt}{\makebox(0,0){$\Diamond$}}}
\put(1068,122){\raisebox{-.8pt}{\makebox(0,0){$\Diamond$}}}
\put(1053,230){\raisebox{-.8pt}{\makebox(0,0){$\Diamond$}}}
\put(1039,241){\raisebox{-.8pt}{\makebox(0,0){$\Diamond$}}}
\put(1025,141){\raisebox{-.8pt}{\makebox(0,0){$\Diamond$}}}
\put(1011,150){\raisebox{-.8pt}{\makebox(0,0){$\Diamond$}}}
\put(996,145){\raisebox{-.8pt}{\makebox(0,0){$\Diamond$}}}
\put(982,125){\raisebox{-.8pt}{\makebox(0,0){$\Diamond$}}}
\put(968,224){\raisebox{-.8pt}{\makebox(0,0){$\Diamond$}}}
\put(953,231){\raisebox{-.8pt}{\makebox(0,0){$\Diamond$}}}
\put(939,135){\raisebox{-.8pt}{\makebox(0,0){$\Diamond$}}}
\put(925,149){\raisebox{-.8pt}{\makebox(0,0){$\Diamond$}}}
\put(911,152){\raisebox{-.8pt}{\makebox(0,0){$\Diamond$}}}
\put(896,150){\raisebox{-.8pt}{\makebox(0,0){$\Diamond$}}}
\put(882,141){\raisebox{-.8pt}{\makebox(0,0){$\Diamond$}}}
\put(868,123){\raisebox{-.8pt}{\makebox(0,0){$\Diamond$}}}
\put(854,218){\raisebox{-.8pt}{\makebox(0,0){$\Diamond$}}}
\put(839,223){\raisebox{-.8pt}{\makebox(0,0){$\Diamond$}}}
\put(825,131){\raisebox{-.8pt}{\makebox(0,0){$\Diamond$}}}
\put(811,145){\raisebox{-.8pt}{\makebox(0,0){$\Diamond$}}}
\put(796,151){\raisebox{-.8pt}{\makebox(0,0){$\Diamond$}}}
\put(782,153){\raisebox{-.8pt}{\makebox(0,0){$\Diamond$}}}
\put(768,151){\raisebox{-.8pt}{\makebox(0,0){$\Diamond$}}}
\put(754,147){\raisebox{-.8pt}{\makebox(0,0){$\Diamond$}}}
\put(739,138){\raisebox{-.8pt}{\makebox(0,0){$\Diamond$}}}
\put(725,122){\raisebox{-.8pt}{\makebox(0,0){$\Diamond$}}}
\put(711,214){\raisebox{-.8pt}{\makebox(0,0){$\Diamond$}}}
\put(697,218){\raisebox{-.8pt}{\makebox(0,0){$\Diamond$}}}
\put(682,127){\raisebox{-.8pt}{\makebox(0,0){$\Diamond$}}}
\put(668,142){\raisebox{-.8pt}{\makebox(0,0){$\Diamond$}}}
\put(654,149){\raisebox{-.8pt}{\makebox(0,0){$\Diamond$}}}
\put(640,152){\raisebox{-.8pt}{\makebox(0,0){$\Diamond$}}}
\put(625,153){\raisebox{-.8pt}{\makebox(0,0){$\Diamond$}}}
\put(611,152){\raisebox{-.8pt}{\makebox(0,0){$\Diamond$}}}
\put(597,149){\raisebox{-.8pt}{\makebox(0,0){$\Diamond$}}}
\put(582,144){\raisebox{-.8pt}{\makebox(0,0){$\Diamond$}}}
\put(568,136){\raisebox{-.8pt}{\makebox(0,0){$\Diamond$}}}
\put(554,121){\raisebox{-.8pt}{\makebox(0,0){$\Diamond$}}}
\put(540,210){\raisebox{-.8pt}{\makebox(0,0){$\Diamond$}}}
\put(525,213){\raisebox{-.8pt}{\makebox(0,0){$\Diamond$}}}
\put(511,125){\raisebox{-.8pt}{\makebox(0,0){$\Diamond$}}}
\put(497,139){\raisebox{-.8pt}{\makebox(0,0){$\Diamond$}}}
\put(483,146){\raisebox{-.8pt}{\makebox(0,0){$\Diamond$}}}
\put(468,151){\raisebox{-.8pt}{\makebox(0,0){$\Diamond$}}}
\put(454,153){\raisebox{-.8pt}{\makebox(0,0){$\Diamond$}}}
\put(440,153){\raisebox{-.8pt}{\makebox(0,0){$\Diamond$}}}
\put(425,153){\raisebox{-.8pt}{\makebox(0,0){$\Diamond$}}}
\put(411,151){\raisebox{-.8pt}{\makebox(0,0){$\Diamond$}}}
\put(397,148){\raisebox{-.8pt}{\makebox(0,0){$\Diamond$}}}
\put(383,142){\raisebox{-.8pt}{\makebox(0,0){$\Diamond$}}}
\put(368,134){\raisebox{-.8pt}{\makebox(0,0){$\Diamond$}}}
\put(354,120){\raisebox{-.8pt}{\makebox(0,0){$\Diamond$}}}
\put(340,208){\raisebox{-.8pt}{\makebox(0,0){$\Diamond$}}}
\put(326,210){\raisebox{-.8pt}{\makebox(0,0){$\Diamond$}}}
\put(311,123){\raisebox{-.8pt}{\makebox(0,0){$\Diamond$}}}
\put(297,137){\raisebox{-.8pt}{\makebox(0,0){$\Diamond$}}}
\put(283,144){\raisebox{-.8pt}{\makebox(0,0){$\Diamond$}}}
\put(268,149){\raisebox{-.8pt}{\makebox(0,0){$\Diamond$}}}
\put(254,151){\raisebox{-.8pt}{\makebox(0,0){$\Diamond$}}}
\put(240,153){\raisebox{-.8pt}{\makebox(0,0){$\Diamond$}}}
\put(226,153){\raisebox{-.8pt}{\makebox(0,0){$\Diamond$}}}
\put(211,153){\raisebox{-.8pt}{\makebox(0,0){$\Diamond$}}}
\put(197,152){\raisebox{-.8pt}{\makebox(0,0){$\Diamond$}}}
\put(140.0,82.0){\rule[-0.200pt]{240.659pt}{0.400pt}}
\put(1139.0,82.0){\rule[-0.200pt]{0.400pt}{115.150pt}}
\put(140.0,560.0){\rule[-0.200pt]{240.659pt}{0.400pt}}
\put(140.0,82.0){\rule[-0.200pt]{0.400pt}{115.150pt}}
\end{picture}
\end{center}
\end{figure}
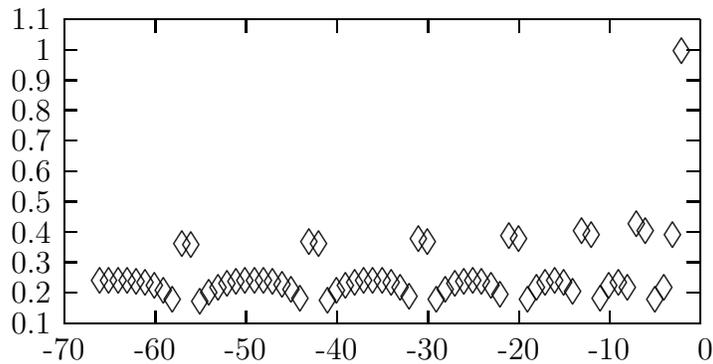

Before proving the lemma, we will justify our claim that one may, for any $\epsilon>0$, find a value of $c$ and an $\alpha\in\QQ$ such that
$$0<\h_\phi(\alpha)<\left(\frac{1}{8}+\epsilon\right)\log|c|.$$

Let $k\in\ZZ$ be positive, and let $c=-k^2-k+1$.  One may verify that $|-3k+2|>2|c|^{1/2}$ for all $k\geq 1$, and so by Lemma~\ref{lem:arch:bcalv}, we have
$$\h_\phi(-3k+2)\leq \log|-3k+2|+\log \frac{3}{2}.$$
Thus $\h_\phi(-3k+2)= \frac{1}{2}\log|c|+O(1)$.  Noting that $\phi^2(k)=-3k+2$, we have
$$\h_\phi(k)=\frac{1}{8}\log|c|+O(1).$$

\begin{proof}[Proof of Lemma~\ref{lem:computations}]  We will in fact show that the estimate holds for $\phi(z)=z^d+c$, where $d\geq 2$ and $c\in\ZZ\setminus\{0\}$.  Note that, by the discussion above, it suffices to show that
$$\left|\h_\phi(\alpha)-h(\alpha)\right|\leq \log|2c|.$$

Write $\alpha=a/b$, with $a$ and $b$ coprime integers.  Note first that
$$\phi(\alpha)=\frac{a^d+b^dc}{b^d}$$
is expressed in lowest terms.  Thus, in particular, 
$$h(\phi(\alpha))=\log\max\{b^d, |a^d+b^dc|\}.$$
It is clear that this is bounded above by
$$d\log\max\{|a|, |b|\}+\log|2c|.$$
The lower bound is slightly more troublesome.  If $|b|\geq |a|$ then
$$h(\phi(\alpha))\geq \log|b^d|=dh(\alpha),$$
so suppose $|a| > |b|$.  If $|a^d+b^dc|\geq\frac{1}{2}a^d$, then
$$h(\phi(\alpha))\geq \log|a^d+b^dc|\geq dh(\alpha)-\log 2.$$
If $\frac{1}{2}a^d<|a^d+b^dc|$ then $|b^dc|>\frac{1}{2}a^d$, and so
$$h(\phi(\alpha))\geq d\log|b|>d\log|a|-\log|2c|=dh(\alpha)-\log|2c|.$$

Thus we've shown that
$$\left|h(\alpha)-\frac{1}{d}h(\phi(\alpha))\right|\leq \frac{\log|2c|}{d}.$$
We may now compute the standard telescoping sum to estimate the canonical height:
\begin{eqnarray*}
\left|h(\alpha)-\frac{1}{d^k}h(\phi^k(\alpha))\right|&\leq&\sum_{n=0}^{k-1}\frac{1}{d^n}\left|h(\phi^n(\alpha))-\frac{1}{d}h(\phi^{n+1}(\alpha))\right|\\
&\leq& \left(1+\frac{1}{d}+\cdots+\frac{1}{d^k}\right)\frac{\log|2c|}{d}.
\end{eqnarray*}
Letting $k\rightarrow\infty$, we obtain
$$\left|\h_\phi(\alpha)-h(\alpha)\right|\leq \frac{\log|2c|}{d-1}.$$
\end{proof}

\vspace{3mm}
\noindent\textsc{Department of Mathematics\\ University of Toronto\\ Toronto, Canada}

\end{document}